\documentclass[leqno,12pt]{amsart}

\usepackage{amsmath,amsfonts,amsthm,mathrsfs,amssymb}
\usepackage[all]{xy}
\usepackage{enumerate}
\usepackage{graphicx,epsfig}
\usepackage{hyperref}
\usepackage{mathtools}
\usepackage{bm}
\usepackage{tikz}
\usepackage{pgf}
\usepackage{epsfig}

\usetikzlibrary{matrix}

\theoremstyle{plain}
\newtheorem{thm}{Theorem}[section]
\newtheorem{cor}[thm]{Corollary}
\newtheorem{lem}[thm]{Lemma}

\newtheorem{prop}[thm]{Proposition}

\theoremstyle{definition}
\newtheorem{defn}[thm]{Definition}

\newtheorem{rem}[thm]{Remark}

\newcommand\NN{\mathbb N}

\newcommand\bq{{\underline{q}}}
\newcommand\bb{{\underline{\bullet}}}

\newcommand\Cc{\mathcal C}

\def\ZZ{{\mathbb Z}}

\def\C.{C_\bullet}
\def\F.{F_\bullet}
\def\k.{\mathcal{K}_{\bullet}}

\newcommand\ra{\rightarrow}

\newcommand\intF{\textnormal{I}}
\newcommand\compF{\textnormal{U}}
\newcommand\Pface{\textnormal{P}}
\newcommand\Face{\textnormal{F}}

\newcommand\coker{\textnormal{coker}}

\newcommand\Spec{\textnormal{Spec}}

\title{Homology of multiple complexes and Mayer-Vietoris spectral sequences}

\author{Marc Chardin}
\address{Institut de Mathématiques de Jussieu, CNRS \& Sorbonne Université, France}
\email{marc.chardin@imj-prg.fr, Marc.CHARDIN@cnrs.fr}

\author{Rafael Holanda}
\address{Departamento de Matemática, Universidade Federal da Paraíba - 58051-900, João Pessoa, PB, Brazil}
\email{rfh@academico.ufpb.br, rf.holanda@gmail.com}

\author{José Naéliton}
\address{Departamento de Matemática, Universidade Federal da Paraíba - 58051-900, João Pessoa, PB, Brazil}
\email{jnaeliton@gmail.com}

\begin{document}

\maketitle

\begin{abstract}
Similarities are noted in two Mayer-Vietoris spectral sequences that generalize to any number of ideals the Mayer-Vietoris exact sequence in local cohomology for two ideals. One has as first terms \v{C}ech cohomology with respect to sums of the given ideals and converge to cohomology with respect to the product of the ideals, the other has as first terms  \v{C}ech cohomology with respect to products of the given ideals and converge to cohomology with respect to the sum of the ideals. The first one was obtained by Lyubeznik in \cite{Lyu}, while the second is constructed in \cite[Chapter 2]{Hol} and could also be deduced from results in \cite{God}. We present results on the cohomology of multiple complexes that enables  to deduce both from two related constructions on multiple complexes. A key ingredient is a fact that seems not to have been noticed before: cohomology with respect to a product of ideals is the one of a subcomplex of the \v{C}ech complex computing  cohomology with respect to the sum of the given ideals; this provides a much shorter complex to compute cohomology with respect to product of ideals.
\end{abstract}

\section{Introduction}

Let $R$ be a commutative unitary ring, $I$ a finitely generated ideal of $R$ and $M$ an $R$-module. \v{C}ech cohomology modules $H^i_I(M)$ are important objects  in commutative algebra and algebraic geometry. Their vanishing is studied throughout combinatorial properties in \cite{AGZ,Lyu}, as well as the multigraded pieces of such modules are investigated in \cite{BC, Hol} (see also the references therein). Spectral sequences play a fundamental role in these works.

Similarities in the spectral sequences in \cite[Chapter 2]{Hol} and \cite{Lyu} motivated us to search for a framework where both may take place. Multicomplex theory as presented in \cite{Ver} is one. We avoid the geometric or combinatorial taste given in \cite{AGZ,God, Lyu} by approaching the spectral sequences in a more elementary way. 

There are two main spectral sequences, each one of these in two variants. They all compare cohomologies of some natural subquotient complexes attached to faces of the cone ${\mathbb R}_{\geq 0}^n$ and their interiors, for a $n$-multicomplex. To give an idea of the picture: one spectral sequence has first terms corresponding to complexes obtained from the interior of faces of all dimensions and converge to the cohomology of the total complex, while the other has first terms corresponding to the faces of all dimensions and converges to the homology of the interior of the complex.

In Section \ref{multcomp}, we develop constructions on multicomplexes and give spectral sequences out of these constructions, in order to prove Theorem \ref{fourspec}, our main result in the theory of multicomplexes. Section \ref{cohomologysection} relates cohomology of multiple \v{C}ech complexes and its interior to local cohomology with respect to sum and product of the corresponding ideals, see Theorem \ref{quasiisomorphism}. As an application of Theorem \ref{fourspec} and Theorem \ref{quasiisomorphism}, we provide four Mayer-Vietoris spectral sequences in Section \ref{fourmvss}; two of these coincide with the ones mentioned above \cite{Hol, Lyu}, they are as follows:

\begin{thm}
Let $R$ be a commutative unitary ring, $M$ be an $R$-module and $\mathfrak{a}_i$, for $1\leq i\leq n$, be finitely generated ideals. There exist two converging spectral sequences 
\begin{itemize}

\item [(1)] $E_1^{n-p,q}=\bigoplus_{i_1<\ldots<i_p}^{1\leq p\leq n}H^q_{\mathfrak{a}_{i_1}+\cdots+\mathfrak{a}_{i_p}}(M)\Rightarrow_p H^{q-(p-1)}_{\mathfrak{a}_1\cdots\mathfrak{a}_n}(M)$,\medskip

\item [(2)] $E_1^{p,q}=\bigoplus_{i_1<\ldots<i_p}^{1\leq p\leq n}H^{q}_{\mathfrak{a}_{i_1}\cdots\mathfrak{a}_{i_p}}(M)\Rightarrow_p H^{q+(p-1)}_{\mathfrak{a}_1+\cdots +\mathfrak{a}_n}(M)$.
\end{itemize}
\end{thm}

The similarities between these two spectral sequences relating \v Cech cohomologies was the starting point of our work.

\section{Multicomplexes and spectral sequences}\label{multcomp}

\subsection{Setup}

We begin this section recalling the formal definition of a multicomplex in an abelian category; see \cite[Chap. I,  \S 2]{Ver} for more details.

\begin{defn}\label{defMC}
Let $n$ be a positive integer. A commutative (resp. anticommutative) $n$-multicomplex $C^\bb$ is a family of objects $C^\bq$ for all $\bq=(q_1,\ldots,q_n)\in\ZZ^n=\oplus_{i=1}^{n}\ZZ e_i$ together with a family of homomorphisms
$d^{\bq,i}: C^\bq\rightarrow C^{\bq+e_i}$
for all $i=1,\ldots,n$ such that conditions (a) and (b) (resp. (a) and (b')) are satisfied :
\begin{itemize}
    \item [(a)] $d^{\bq+e_i,i}\circ d^{\bq,i}=0$ for all $\bq\in\ZZ^n$ and $i=1,\ldots,n$;
    \item [(b)] $d^{\bq+e_j,i}\circ d^{\bq,j}=d^{\bq+e_i,j}\circ d^{\bq,i}$ for all $\bq\in\ZZ^n$ and $i,j=1,\ldots,n$;
        \item [(b')] $d^{\bq+e_j,i}\circ d^{\bq,j}=-d^{\bq+e_i,j}\circ d^{\bq,i}$ for all $\bq\in\ZZ^n$ and $i,j=1,\ldots,n$.
\end{itemize}
\end{defn}

There are many ways to transform a commutative multicomplex into an anticommutative one, and vice-versa. The transformation $\sigma : d^{\bq,i}\mapsto (-1)^{q_1+\ldots+q_{i-1}}d^{\bq,i}$ in \cite{Ver} is one such option; it satisfies $\sigma \circ \sigma (C^\bb ) =C^\bb$.


As in the case of double complexes, the total complex of a multicomplex is defined,

\begin{defn}
Let $C^\bb$ be a  $n$-multicomplex. The total complex (or totalization) $T(C^\bb)^\bullet$ of $C^\bb$ is defined by
\begin{itemize}
\item [(a)] $T(C^\bb)^m=\bigoplus_{q_1+\ldots+q_n=m}C^\bq,\ \forall m\geq0;$ 
\end{itemize}
$d^m:T(C^\bb)^m \rightarrow T(C^\bb)^{m+1}$  for $x\in C^\bq$ with $q_1+\ldots+q_n=m$ is given by
\begin{itemize}
\item [(b)] $d^m(x):=\sum_{i=1}^n \sigma (d^{\bq,i})(x)$ if $C^\bb$ is commutative;
\item [(b')] $d^m(x):=\sum_{i=1}^n d^{\bq,i}(x)$ if $C^\bb$ is anticommutative.
\end{itemize}
\end{defn}

It is a routine exercise to verify that $T(C^\bb)^\bullet$ is indeed a complex.

Despite the possibilities of sign changes on the differentials that keeps the property of being a multiple complex, it is important to notice that every kernel or cokernel of these maps, as well as objects constructed from these -- typically ones obtained at pages of some spectral sequences in our work -- are unchanged, up to the sign of some maps involved. In practice, depending on the situation, one may for instance choose a commutative version or an anticommutative version in order to make a proof more transparent by avoiding part of the sign tracking that could obscure the proof.

Let $C^{\underline{\bullet}}$ be a $n$-multicomplex with components $C^\bq$ verifying $C^\bq=0$ for $\bq\not\in \NN^n=\oplus_{i=1}^{n}\NN e_i$,  with $e_i$ from Definition \ref{defMC}. For such a multicomplex, we will consider several complexes attached to the faces of the polyhedral cone ${\mathbb R}_{\geq 0}^n$ :

For $i_1<\cdots <i_p$, let 

-- $\Face_{i_1,\ldots ,i_p}:=\NN e_{i_1}\oplus \cdots \oplus \NN e_{i_p}$, denote a $p$-dimensional face, 


-- $\Pface_{i_1,\ldots ,i_p} := \Face_{i_1,\ldots ,i_p}\setminus \underline 0$,  the corresponding punctured $p$-dimensional face, 

-- $\intF_{i_1,\ldots ,i_p}:=\{ \bq\in \NN^n\ \vert \ q_i= 0\Leftrightarrow i\not\in \{ i_1,\ldots ,i_p\}\}$,  the interior of this $p$-dimensional face,


-- $\compF_{i_1,\ldots ,i_p} :=\NN^n\setminus \Face_{i_1,\ldots ,i_p} $,  the complement of this $p$-dimensional face.

For instance, when $n=2$, one has for example (full dots show corresponding sets) :

$
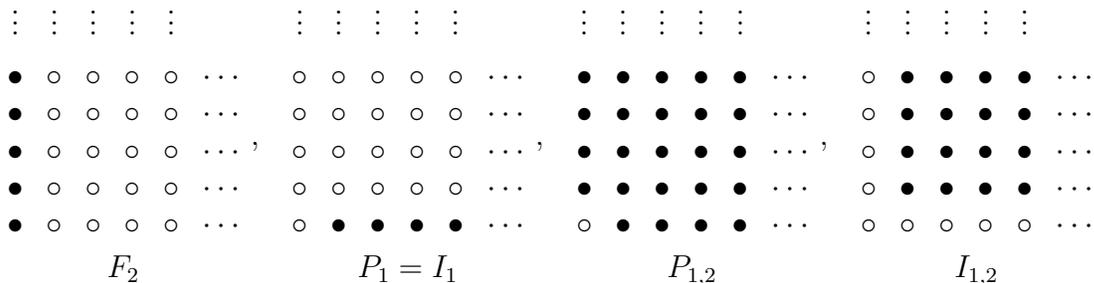
\begin{matrix}
{\xymatrix@R=0pt@C=0pt{
\vdots&\vdots&\vdots&\vdots&\vdots&\\
\bullet&\circ&\circ&\circ&\circ&\cdots\\
\bullet&\circ&\circ&\circ&\circ&\cdots\\
\bullet&\circ&\circ&\circ&\circ&\cdots\\
\bullet&\circ&\circ&\circ&\circ&\cdots\\
\bullet&\circ&\circ&\circ&\circ&\cdots\\}}\\
F_2
\end{matrix}
$, \  
$
\begin{matrix}
{\xymatrix@R=0pt@C=0pt{
\vdots&\vdots&\vdots&\vdots&\vdots&\\
\circ&\circ&\circ&\circ&\circ&\cdots\\
\circ&\circ&\circ&\circ&\circ&\cdots\\
\circ&\circ&\circ&\circ&\circ&\cdots\\
\circ&\circ&\circ&\circ&\circ&\cdots\\
\circ&\bullet&\bullet&\bullet&\bullet&\cdots\\}}\\
P_1=I_1
\end{matrix}
$, \ 
$
\begin{matrix}
{\xymatrix@R=0pt@C=0pt{
\vdots&\vdots&\vdots&\vdots&\vdots&\\
\bullet&\bullet&\bullet&\bullet&\bullet&\cdots\\
\bullet&\bullet&\bullet&\bullet&\bullet&\cdots\\
\bullet&\bullet&\bullet&\bullet&\bullet&\cdots\\
\bullet&\bullet&\bullet&\bullet&\bullet&\cdots\\
\circ&\bullet&\bullet&\bullet&\bullet&\cdots\\}}\\
P_{1,2}
\end{matrix}
$, \ 
$
\begin{matrix}
{\xymatrix@R=0pt@C=0pt{
\vdots&\vdots&\vdots&\vdots&\vdots&\\
\circ&\bullet&\bullet&\bullet&\bullet&\cdots\\
\circ&\bullet&\bullet&\bullet&\bullet&\cdots\\
\circ&\bullet&\bullet&\bullet&\bullet&\cdots\\
\circ&\bullet&\bullet&\bullet&\bullet&\cdots\\
\circ&\circ&\circ&\circ&\circ&\cdots\\}}\\
I_{1,2}
\end{matrix}
$

These sets could as well be defined in terms of coordinates that are zero on these faces, namely $\Face_{i_1,\ldots ,i_p}^*:=\Face_{\{ 1,\ldots ,n\}\setminus \{ i_1,\ldots ,i_p\} }$ and similarly for the three other sets.

Notice that $\Face_{1,\ldots ,n}=\Face^*_\emptyset =\NN^n$ and  $\Face_\emptyset =\Face^*_{1,\ldots ,n}=\underline 0$.

Write $C^{\bb}_{\compF_{i_1,\ldots ,i_p}}$ for the subcomplex of $C^{\bb}$ obtained by replacing $C^\bq$ by the zero module unless $\bq\in \compF_{i_1,\ldots ,i_p}$. It is indeed a subcomplex since $\compF_{i_1,\ldots ,i_p}+e_i \subseteq \compF_{i_1,\ldots ,i_p}$ for any $i\geq1$. 

Similarly, we define $C^{\bb}_{\Face_{i_1,\ldots ,i_p}}$ and notice that $C^\bb_{\Face_{i_1,\ldots,i_p}}=C^\bb/C^\bb_{\compF_{i_1,\ldots,i_p}}$ is a quotient of $C^\bb$. 

We denote by $C^{\bb}_{\Pface_{i_1,\ldots ,i_p}}$  and  $C^{\bb}_{\intF_{i_1,\ldots ,i_p}}$ the subcomplexes of $C^{\bb}_{\Face_{i_1,\ldots ,i_p}}$ obtained by replacing by zero the modules for $\bq\not\in \Pface_{i_1,\ldots ,i_p}$ and  $\bq\not\in \intF_{i_1,\ldots ,i_p}$, respectively.

We finally define an augmented version ${^+{C}}^{\bb}_{\intF_{i_1,\ldots ,i_p}}$ of $C^{\bb}_{\intF_{i_1,\ldots ,i_p}}$ for $p>0$ by adding as module $C^{\underline 0}$ in homological degree $e_{i_1}+\cdots +e_{i_{p-1}}$ and adding the map
$$
\xymatrix{
C^{\underline 0}[e_{i_1}+\cdots +e_{i_{p-1}}]\ar^(.55){ d^{e_{i_1}+\cdots +e_{i_{p-1}},i_{p} }\circ \;\cdots \;\circ \; d^{e_{i_{1}},i_{2}}\; \circ \; d^{\underline 0,i_1}}[rrrrrr]&&&&&&C^{e_{i_1}+\cdots +e_{i_{p}}}.\\
}
$$
Notice that  $C^{\bb}_{\intF_{i_1,\ldots ,i_p}}$ starts in total homological degree $p$ with unique summand $C^{e_{i_1}+\cdots +e_{i_{p}}}$.






The totalizations of these complexes will be denoted respectively by $C^{\bullet}_{\Face_{i_1,\ldots ,i_p} }$, $C^{\bullet}_{\Pface_{i_1,\ldots ,i_p} }$, $C^{\bullet}_{\intF_{i_1,\ldots ,i_p}}$ and ${^+{C}}^{\bullet}_{\intF_{i_1,\ldots ,i_p}}$. These start respectively in homological degree 0 or higher, 1 or higher, $p$ or higher and $p-1$ or higher.

As the sets $\Face^*_{i_1,\ldots ,i_p}$, $\Pface^*_{i_1,\ldots ,i_p}$, $\intF^*_{i_1,\ldots ,i_p}$ or $\compF^*_{i_1,\ldots ,i_p}$ are respectively equal to $\Face_{j_1,\ldots ,j_{n-p}}$, $\Pface_{j_1,\ldots ,j_{n-p}}$, $\intF_{j_1,\ldots ,j_{n-p}}$ or $\compF_{j_1,\ldots ,j_{n-p}}$ for $ \{ j_1,\ldots ,j_{n-p}\}:=\{ 1,\ldots ,n\}\setminus \{ i_1,\ldots ,i_p\} $, we may alternatively use this other notation.

Finally, ${\Pface}:=\Pface_{1,\ldots ,n}$ and ${\intF}:= {\intF}_{1,\ldots ,n}$.

In Section \ref{cohomologysection}, we will see that these four types of complexes provide natural cohomologies with respect to sums or products for a multiple complex constructed from \v{C}ech complexes with respect to several ideals.

\subsection{Spectral sequences arising from multicomplexes}  

We will now show the following result by constructing the corresponding four spectral sequences:

\begin{thm}\label{fourspec}
Let $C^{\underline{\bullet}}$ be $n$-multicomplex satisfying $C^\bq=0$ for $\bq\not\in \NN^n=\oplus_{i=1}^{n}\NN e_i$.
Then there exist four convergent spectral sequences as follows:

\begin{itemize}
\item[(1)] $E^{p,q}_1=\oplus_{i_1<\cdots <i_p}H^q(C^{\bullet}_{\Face^*_{i_1,\ldots ,i_p}})\Rightarrow H^{p+q}(C^{\bullet}_{\intF})$,

\item[(2)] $E^{p,q}_1=\oplus_{{i_1<\cdots <i_p}\atop{p\not= n}}H^q(C^{\bullet}_{\Face^*_{i_1,\ldots ,i_p}})\Rightarrow H^{p+q}({^+{C}}^{\bullet}_{\intF})$,

\item[(3)] $E_1^{p,q}=\oplus_{i_1<\ldots<i_{p}}H^{p+q}(C^{\bullet}_{\intF_{i_1,\ldots ,i_p} })\Rightarrow H^{p+q}(C^{\bullet}_{\Pface})$,

\item[(4)] $E_1^{p,q}=\oplus_{i_1<\ldots<i_{p}}H^{p+q}({^+{C}}^{\bullet}_{\intF_{i_1,\ldots ,i_p}})\Rightarrow H^{p+q}(C^{\bullet})$.
\end{itemize}

\end{thm}

These will follow from natural filtrations on simple explicit constructions from $C^{\underline{\bullet}}$, or on $C^{\underline{\bullet}}$ itself. The rest of this section is devoted to detail these constructions.\\

Our four Mayer-Vietoris spectral sequences will be direct corollaries of these, in view of the results of Section \ref{cohomologysection}.\\

We start by a construction that will be used for the two first spectral sequences.

Write $D^{\bullet,\bb}$ for the subcomplex of the $(n+1)$-multicomplex  $K^\bullet (1,\ldots ,1;C^{\bb})$, where $1$ stands for the identity map, with components the subcomplexes 
$$
D^{p,\bb}:=\bigoplus_{i_1,\ldots ,i_p}C^{\bb}_{\compF_{i_1,\ldots ,i_p}^*}e_{i_1}\wedge\cdots \wedge e_{i_p}\subseteq K^p({{\underbrace{1,\ldots ,1}_{n\ {\rm times}}}};C^{\bb})=\bigoplus_{i_1,\ldots ,i_p}C^{\bb}e_{i_1}\wedge\cdots \wedge e_{i_p}.
$$
where $K^\bullet(1,\ldots,1; - )$ stands for the Koszul complex of sequence $1,\ldots,1$ -- such a Koszul complex is exact, since $n\geq 1$, by \cite[1.6.4]{BH}.

Finally we set $Q^{\bullet,\bb}:=\coker (D^{\bullet,\bb}\ra K^\bullet (1,\ldots ,1;C^{\bb}))$. 

Notice that $Q^{p,\bb}=\bigoplus_{i_1,\ldots ,i_p}C^{\bb}_{\Face^*_{i_1,\ldots ,i_p}}e_{i_1}\wedge\cdots \wedge e_{i_p}$; in other words $Q^{p,\bb}$ is the direct sum of the $(n-p)$-multicomplexes  induced on  $\Face^*_{i_1,\ldots ,i_p}$ by restriction.\\

\textbf{Illustration for the case of double complexes:}

$
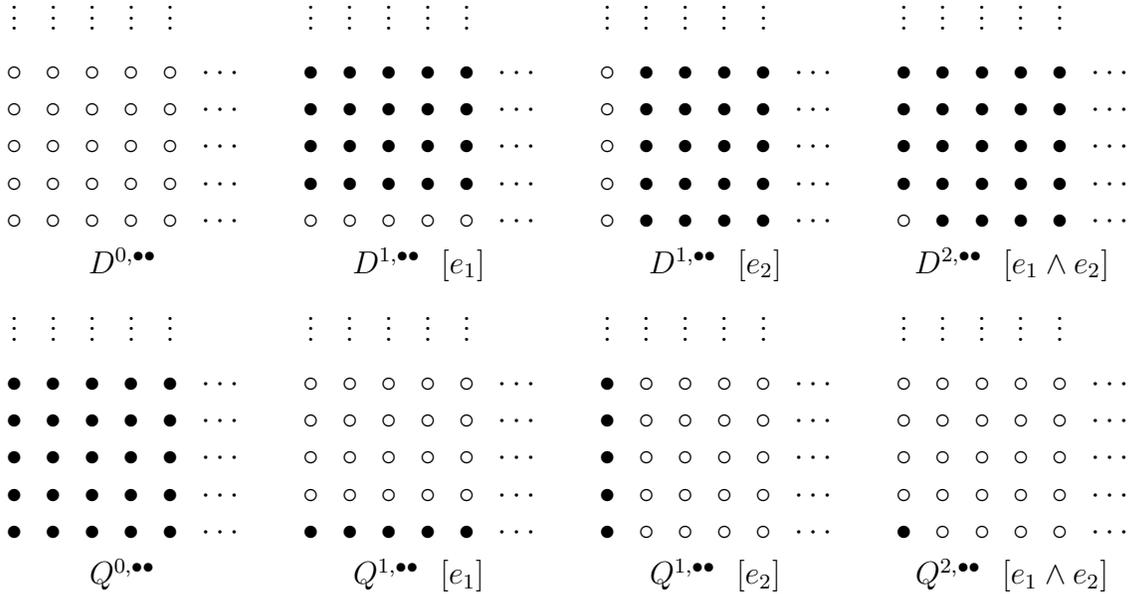
\begin{matrix}
{\xymatrix@R=0pt@C=0pt{
\vdots&\vdots&\vdots&\vdots&\vdots\\
\circ&\circ&\circ&\circ&\circ&\cdots\\
\circ&\circ&\circ&\circ&\circ&\cdots\\
\circ&\circ&\circ&\circ&\circ&\cdots\\
\circ&\circ&\circ&\circ&\circ&\cdots\\
\circ&\circ&\circ&\circ&\circ&\cdots\\}}\\
D^{0,\bullet\bullet}\\
\end{matrix}
$
\ \ \ 
$
\begin{matrix}
{\xymatrix@R=0pt@C=0pt{
\vdots&\vdots&\vdots&\vdots&\vdots&\\
\bullet&\bullet&\bullet&\bullet&\bullet&\cdots\\
\bullet&\bullet&\bullet&\bullet&\bullet&\cdots\\
\bullet&\bullet&\bullet&\bullet&\bullet&\cdots\\
\bullet&\bullet&\bullet&\bullet&\bullet&\cdots\\
\circ&\circ&\circ&\circ&\circ&\cdots\\}}\\
D^{1,\bullet\bullet}\ \ [e_1]\\
\end{matrix}
$
\ \ \ 
$
\begin{matrix}
{\xymatrix@R=0pt@C=0pt{
\vdots&\vdots&\vdots&\vdots&\vdots&\\
\circ&\bullet&\bullet&\bullet&\bullet&\cdots\\
\circ&\bullet&\bullet&\bullet&\bullet&\cdots\\
\circ&\bullet&\bullet&\bullet&\bullet&\cdots\\
\circ&\bullet&\bullet&\bullet&\bullet&\cdots\\
\circ&\bullet&\bullet&\bullet&\bullet&\cdots\\}}\\
D^{1,\bullet\bullet}\ \ [e_2]\\
\end{matrix}
$
\ \ \ 
$
\begin{matrix}
{\xymatrix@R=0pt@C=0pt{
\vdots&\vdots&\vdots&\vdots&\vdots&\\
\bullet&\bullet&\bullet&\bullet&\bullet&\cdots\\
\bullet&\bullet&\bullet&\bullet&\bullet&\cdots\\
\bullet&\bullet&\bullet&\bullet&\bullet&\cdots\\
\bullet&\bullet&\bullet&\bullet&\bullet&\cdots\\
\circ&\bullet&\bullet&\bullet&\bullet&\cdots\\}}\\
D^{2,\bullet\bullet}\ \ [e_1\wedge e_2]\\
\end{matrix}
$

$
\begin{matrix}
{\xymatrix@R=0pt@C=0pt{
\vdots&\vdots&\vdots&\vdots&\vdots\\
\bullet&\bullet&\bullet&\bullet&\bullet&\cdots\\
\bullet&\bullet&\bullet&\bullet&\bullet&\cdots\\
\bullet&\bullet&\bullet&\bullet&\bullet&\cdots\\
\bullet&\bullet&\bullet&\bullet&\bullet&\cdots\\
\bullet&\bullet&\bullet&\bullet&\bullet&\cdots\\}}\\
Q^{0,\bullet\bullet}\\
\end{matrix}
$
\ \ \ 
$
\begin{matrix}
{\xymatrix@R=0pt@C=0pt{
\vdots&\vdots&\vdots&\vdots&\vdots&\\
\circ&\circ&\circ&\circ&\circ&\cdots\\
\circ&\circ&\circ&\circ&\circ&\cdots\\
\circ&\circ&\circ&\circ&\circ&\cdots\\
\circ&\circ&\circ&\circ&\circ&\cdots\\
\bullet&\bullet&\bullet&\bullet&\bullet&\cdots\\}}\\
Q^{1,\bullet\bullet}\ \ [e_1]\\
\end{matrix}
$
\ \ \ 
$
\begin{matrix}
{\xymatrix@R=0pt@C=0pt{
\vdots&\vdots&\vdots&\vdots&\vdots&\\
\bullet&\circ&\circ&\circ&\circ&\cdots\\
\bullet&\circ&\circ&\circ&\circ&\cdots\\
\bullet&\circ&\circ&\circ&\circ&\cdots\\
\bullet&\circ&\circ&\circ&\circ&\cdots\\
\bullet&\circ&\circ&\circ&\circ&\cdots\\}}\\
Q^{1,\bullet\bullet}\ \ [e_2]\\
\end{matrix}
$
\ \ \ 
$
\begin{matrix}
{\xymatrix@R=0pt@C=0pt{
\vdots&\vdots&\vdots&\vdots&\vdots&\\
\circ&\circ&\circ&\circ&\circ&\cdots\\
\circ&\circ&\circ&\circ&\circ&\cdots\\
\circ&\circ&\circ&\circ&\circ&\cdots\\
\circ&\circ&\circ&\circ&\circ&\cdots\\
\bullet&\circ&\circ&\circ&\circ&\cdots\\}}\\
Q^{2,\bullet\bullet}\ \ [e_1\wedge e_2]\\
\end{matrix}
$

\begin{prop}\label{mainspectral} With notations as above, the following holds:
\begin{itemize}
\item [(a)] For any $\bq\in \NN^n$, $H^p (D^{\bullet ,\bq})=0$ for $p\not=1$, $H^p (Q^{\bullet ,\bq})=0$ for $p\not= 0$ and 
$$
H^1 (D^{\bullet ,\bq})=H^0 (Q^{\bullet ,\bq})=C^{\bq}_{\intF}.
$$

\item [(b)] Let $Q^\bullet$ be the totalization of $Q^{\bullet,\bb}$, then 
$$
H^i (Q^\bullet)\simeq H^{i}(C^{\bullet}_{\intF}), \forall i.
$$

\item [(c)] There is a spectral sequence,
$$
E^{p,q}_1=\oplus_{i_1<\cdots <i_p}H^q(C^{\bullet}_{\Face^*_{i_1,\ldots ,i_p}})\Rightarrow H^{p+q}(C^{\bullet}_{\intF} ).
$$
\end{itemize}
\end{prop}

\begin{proof}
Since $K^\bullet (1,\ldots ,1;C^{\bq })$ is exact for any $\bq$ (as $n\geq 1$),  (a) is equivalent to the exactness of $Q^{\bullet ,\bq}$ for $\bq\not\in \intF$, because if $\bq\in \intF$ then $Q^{p ,\bq}=0$ unless $p=0$ and $Q^{0 ,\bq}=C^{\bq}$.

Assume that $\bq$ has exactly $t\geq 1$ coordinates equal to zero and $\bq \in \Face^*_{j_1,\ldots ,j_t}$. Then $Q^{\bullet ,\bq}$ is the exact subcomplex 
$$
K^p({{\underbrace{1,\ldots ,1}_{t\ {\rm times}}}};C^{\bq})\subseteq K^p({{\underbrace{1,\ldots ,1}_{n\ {\rm times}}}};C^{\bq})
$$
that corresponds to summands indexed by $e_{i_1}\wedge\cdots \wedge e_{i_p}$ for $\{ i_1,\ldots ,i_p \}\subseteq \{ j_1,\ldots ,j_t \}$.

For (b), denote by $Q^{\bullet ,\bullet}$ the double complex obtained by totalizing along $\NN^n$ the complex $Q^{\bullet,\bb}$. By (a), for any $q\in \NN$, $H^p(Q^{\bullet ,q})=0$ for $p\not= 0$ and $H^0 (Q^{\bullet ,q})=C^{q}_{\intF}$; hence $Q^\bullet$ is quasi-isomorphic to $C^{\bullet}_{\intF}$ and the conclusion follows.

For (c)  recall that $Q^{p,\bb}=\bigoplus_{i_1,\ldots ,i_p}C^{\bb}_{\Face^*_{i_1,\ldots ,i_p}}e_{i_1}\wedge\cdots \wedge e_{i_p}$, hence the second spectral sequence for $Q^{\bullet ,\bullet}$ has first terms $H^{q}(Q^{p,\bullet})=\oplus_{i_1<\cdots <i_p}H^q(C^{\bullet}_{\Face^*_{i_1,\ldots ,i_p}})$ and abuts to the cohomology of $Q^\bullet$ that is in turn  isomorphic to the one of $C^{\bullet}_{\intF}$ by (b).
\end{proof}

\begin{rem}
    
    Alternatively,  one can consider the double complex $F^{p,q}$ defined by $F^{p,q}:=Q^{p,q}$ unless $p=-1$ and $F^{-1,q}:=C^{q}_{\intF}$ and then obtain a spectral sequence
$$
E^{p,q}_1=H^q (F^{p,\bullet})\Rightarrow 0
$$
with $H^q (F^{p,\bullet})=\oplus_{i_1<\cdots <i_p}H^q(C^{\bullet}_{\Face^*_{i_1,\ldots ,i_p}})$ for $p\not= -1$ and $H^q (F^{-1,\bullet})=H^q(C^{\bullet}_{\intF})$.
\end{rem}

The variant of the spectral sequence in part (c) of Proposition \ref{mainspectral} that provides the second spectral sequence in the main theorem is as follows:

\begin{prop}\label{augmentedmainspectral}
With notations as in Proposition \ref{mainspectral}, there is a spectral sequence 
$$
E^{p,q}_1=\oplus_{i_1<\cdots <i_p}H^q(C^{\bullet}_{\Face^*_{i_1,\ldots ,i_p}})\Rightarrow H^{p+q}({^+{C}}^\bullet_\intF)
$$
with $p$ in the range $0\leq p<n$ (i.e. $E^{p,q}_1=0$ for any $p\geq n$ and $q$).
\end{prop}

\begin{proof} Consider the double complex $Q_-^{p,q}$ defined by $Q_-^{p,q}:=Q^{p,q}$ unless $p=n$ and $Q_-^{n,q}:=0$; let $H^q$ denote the $q$-th cohomology of its totalization. It gives rise to a spectral sequence
$$
E^{p,q}_1=\oplus_{i_1<\cdots <i_p}H^q(C^{\bullet}_{\Face^*_{i_1,\ldots ,i_p}})\Rightarrow H^{p+q}.
$$
with $p$ in the range $0\leq p<n$ as claimed. The other spectral sequence has second terms $'E^{n-1,0}_2\simeq C^{\underline{0}}$, $'E^{0,q}_2=H^q (C^{\bullet}_{\intF})$ for any $q$ and $'E^{p,q}_2=0$ otherwise, by Proposition \ref{mainspectral}. As $C^{q}_{\intF}=0$ for $q<n$, $H^q (C^{\bullet}_{\intF})=0$ for $q<n$ and $H^n (C^{\bullet}_{\intF})\subseteq C^{n}_{\intF}=C^{\underline{1}}$. It follows that:

\begin{itemize}
\item[(i)] $H^q =0$ for $q<n-1$,

\item[(ii)]  $H^{n-1}=\ker ('d^{n-1,0}_n : {'E^{n-1,0}_n} \ra {'E}^{0,n}_n )\simeq \ker  (C^{\underline{0}} \ra H^n (C^{\bullet}_{\intF}))$,

\item[(iii)] $H^n=\coker ('d^{n-1,0}_n : {'E^{n-1,0}_n} \ra {'E}^{0,n}_n )\simeq \coker  (C^{\underline{0}} \ra H^n (C^{\bullet}_{\intF}))$,

\item[(iv)] $H^q=H^{q}(C^{\bullet}_{\intF})$ for $q>n$.
\end{itemize}

To conclude the proof, we show that, after the identifications ${'E}^{n-1,0}_n \simeq {'E}^{n-1,0}_1 \simeq C^{\underline{0}}$ and ${'E}^{0,n}_n\simeq {'E}^{0,n}_2 \simeq H^n (C^{\bullet}_{\intF})$, 
one has $'d^{n-1,0}_n =\pm d^{e_{1}+\cdots +e_{n-1},n }\circ \;\cdots \;\circ \; d^{{e_1},2}\; \circ \; d^{\underline0,1}$. We follow the  construction of this spectral sequence in \cite[page 133]{Wei}. 

Write $d_K$ for the Koszul differential and $d_Q$ for the differential on the totalization of the $n$-multicomplex $Q^{\bb}$.  

The identification ${'E}^{n-1,0}_1 \simeq C^{\underline{0}}$ sends the class of an element $x :=\sum_{\vert I\vert =n-1}\alpha_I \wedge_{j\in I}e_j$  to $\sum_i \sum_{\vert I\vert =n-1}\alpha_I e_i\wedge(\wedge_{j\in I}e_j)=\alpha e_1\wedge\cdots\wedge e_n$. Modulo a border,  
$x$ is equal to $\alpha e_2\wedge \cdots \wedge e_n$. 
To determine the image of the class of $x$ by $'d^{n-1,0}_n $, set $\alpha_0:= \alpha$, $x_0:=\alpha_0 e_2\wedge \cdots \wedge e_n$, 
$$
\alpha_q :=d^{e_1+\cdots+e_{q-1},q}\circ\cdots\circ d^{\underline0,1}(\alpha)\in  C^{e_1+\cdots+e_{q}}\subseteq C^{q}
$$  
for $1\leq q\leq n$ and if $1\leq q< n$
$$
x_q:=\alpha_q e_{q+2}\wedge\cdots\wedge e_n \in K^{n-q-1}({{\underbrace{1,\ldots ,1}_{n-q\ {\rm times}}}};C^{e_1+\cdots+e_{q}}).
$$
Now $d_K (x_q)=\varepsilon_q \alpha_q e_{q+1}\wedge e_{q+2}\wedge\cdots\wedge e_n$, with $\varepsilon_q =\pm 1$, since $e_j \wedge e_{q+2}\wedge\cdots\wedge e_n = 0$ for $j\geq q+1$, unless $j=q+1$. 
On the other hand, since by definition $\alpha_q =d^{e_1+\cdots+e_{q-1},q}(\alpha_{q-1})$, it follows that $d_Q (x_{q-1})=\varepsilon'_q\alpha_q e_{q+1}\wedge e_{q+2}\wedge\cdots\wedge e_n$, with $\varepsilon'_q =\pm 1$. 

Hence, setting  $\epsilon_j :=\prod_{i=1}^j-\varepsilon_i\varepsilon'_i$, the element 
$$x':=x_0+\epsilon_1 x_1\oplus \cdots \oplus \epsilon_{n-1} x_{n-1}\in \bigoplus_{q=0}^{n-1}K^{n-q-1}({{\underbrace{1,\ldots ,1}_{n-q\ {\rm times}}}};C^{e_1+\cdots+e_{q}})\subseteq \bigoplus_{q=0}^{n-1}C^{n-q-1,q}
$$ is in $A^{n}_{0}$ as in the construction of \cite[page 133]{Wei}. The conclusion follows, since then, by definition, $'d^{n-1,0}_n (x)=  {'d}^{n-1,0}_n (x')$ is the class of $\pm\alpha_n=\pm d_Q (x_{n-1})\in C^{e_1+\cdots+e_{n}}\subseteq  H^n (C^{\bullet}_{\intF})$.





\end{proof}

\begin{prop}\label{ss3}Let $C^{\underline{\bullet}}$ be a $n$-multicomplex satisfying $C^\bq=0$ for $\bq\not\in \NN^n=\oplus_{i=1}^{n}\NN e_i$. Then there exists a convergent spectral sequence:
 $$E_1^{p,q}=\oplus_{i_1<\ldots<i_p}H^{p+q}(C^{\bullet}_{\intF_{i_1,\ldots ,i_p} })\Rightarrow H^{p+q}(C^{\bullet}_{\NN^n\setminus \{ 0\} }).$$
\end{prop}

\begin{proof}
Given $p\geq 1$, define $$X_p:=\{(q_1,...,q_n)\in\NN^n\ \vert\ \mbox{at most} \ n-p \ \mbox{of the} \ q_j\mbox{'s are zero}\}.$$ 

The family of subcomplexes  $F^\bullet_p$ with $F_p:=\displaystyle\oplus_{\bq\in X_p}C^{\bq}$ is a limited descending filtration of $F^\bullet_1 =C^{\bullet}_{\NN^n\setminus \{ 0\}}$ and therefore yields a spectral sequence converging to $H^{p+q}(C^{\bullet}_{\NN^n\setminus \{ 0\} })$. As for $p\geq 1$,
$F_p/F_{p+1}\simeq\oplus_{i_1<\ldots<i_p}C^{\bullet}_{\intF_{i_1,\ldots ,i_p} }$ (exactly $p$ of the $q_i$'s are not zero), $$E_1^{p,q}=H^{p+q}(F_p/F_{p+1})\simeq\bigoplus_{i_1<...<i_p}H^{p+q}(C^{\bullet}_{\intF_{i_1,\ldots ,i_p} }).$$
\end{proof}

To derive the fourth spectral sequence, let $T^\bb (C^{\underline 0})$ be the trivial hypercube commuting multiple complex on $C^{\underline 0}$ : $T^{\underline q}=C^{\underline 0}$ if $\bq\in \{0,1\} ^n$  and $0$ otherwise, and differentials are the identity if source and target are in degrees that belong to $\{0,1\} ^n$ and $0$ else. If $C^\bb$ is a commuting multiple $\NN^n$ complex, we define a map from $T^\bb (C^{\underline 0})$ to $C^\bb$ by 
$$
\psi_{i_1,\ldots ,i_p}\ :\ x\in T^{e_{i_1}+\cdots +e_{i_p}}(C^{\underline 0})=C^{\underline 0}\mapsto d^{e_{i_1}+\cdots +e_{i_{p-1}},i_{p} }\circ \;\cdots \;\circ \; d^{e_{i_{1}},i_{2}}\; \circ \; d^{0,i_1}(x)\in C^{e_{i_1}+\cdots +e_{i_p}}
$$
and notice that it provides a commuting $\ZZ^{n+1}$-multicomplex ${^\square{C}}^{\bullet, \bb}$ sitting in degrees that belong to $\{-1,0\}\times \NN^n$, with the hypercube sitting in degrees $\{-1\}\times \ZZ^{n}$ and ${^\square{C}}^{0,\bb}=C^\bb$. Denote by $C^\bullet$ the totalization of $C^\bb$.

\begin{prop}Let $C^{\underline{\bullet}}$ be a $n$-multicomplex satisfying $C^\bq=0$ for $\bq\not\in \NN^n=\oplus_{i=1}^{n}\NN e_i$. Then there exists a convergent spectral sequence:
 $$E_1^{p,q}=\oplus_{i_1<\ldots<i_p}H^{p+q}({^{+}C}^{\bullet}_{\intF_{i_1,\ldots ,i_p} })\Rightarrow H^{p+q}(C^{\bullet}).$$
\end{prop}

\begin{proof}
First we may, and will, assume that $C^{\underline{\bullet}}$ is a commutative $n$-multicomplex (applying $\sigma$ as defined in Section \ref{multcomp} if it is anticommutative to go to this case). 
As ${\rm Tot}(T^\bb (C^{\underline 0}))$ has trivial total cohomology,  the complex $T^\bullet:={\rm Tot}({^\square{C}}^{\bullet, \bb})$ satisfies $H^{p+q}(C^{\bullet})\simeq H^{p+q}(T^\bullet)$.

For $p\geq 0$, define $X'_p:=\ZZ \times X_p$, with $X_p$ as in the proof of Proposition  \ref{ss3}.
The family of subcomplexes  $F^\bullet_p$ with $F_p:=\displaystyle\oplus_{\bq'\in X'_p}{^\square{C}}^{\bq '}$ is a limited descending filtration of $F^\bullet_0 =T^\bullet$ and therefore yields a spectral sequence converging to $H^{p+q}(T^\bullet)$. For $p\geq 1$,
$F_p/F_{p+1}$ is the direct sum of the complexes 
$$
\xymatrix{
0\ar[r]&C^{\underline 0}\ar^(.4){\psi_{i_1,\ldots ,i_p}}[rr]&&{C}^{p}_{\intF_{i_1,\ldots ,i_p} }\ar[r]&{C}^{p+1}_{\intF_{i_1,\ldots ,i_p} }\ar[r]&\cdots\\},
$$ hence $$E_1^{p,q}=H^{p+q}(F_p/F_{p+1})\simeq\bigoplus_{i_1<...<i_p}H^{p+q}({^{+}C}^{\bullet}_{\intF_{i_1,\ldots ,i_p} }).$$
For $p=0$, $F_p/F_{p+1}$ is $\xymatrix{0\ar[r]& C^{\underline 0} \ar^{\pm id}[r]& C^{\underline 0}\ar[r]& 0\\}$, since $\psi_{\underline 0}=id$, with $C^{\underline 0}$ sitting in degrees $-1$ and $0$. Hence $E_2^{0,q}=0$ for any $q$ and the conclusion follows.
\end{proof}

\section{Cohomology in multiple \v{C}ech complexes}\label{cohomologysection}


Let $R$ be a commutative unitary ring. For an $R$-module $M$ and $\mathbf{x}:=(x_1,\ldots ,x_m)$ a sequence of elements in $R$, $\Cc^\bullet_\mathbf{x} (M)$ denotes the \v{C}ech complex 
$$
\xymatrix{0\ar[r]&M\ar[r]&\oplus_{i}M_{x_i}\ar[r]&\cdots\ar[r]&M_{x_1\cdots x_m}\ar[r]&0}
$$
and $\check\Cc^\bullet_\mathbf{x} (M)$  the complex obtained by replacing $\Cc^0_\mathbf{x} (M)$ by $0$ in $\Cc^\bullet_\mathbf{x} (M)$ (truncation),
$$
\xymatrix{0\ar[r]&\oplus_{i}M_{x_i}\ar[r]&\cdots\ar[r]&M_{x_1\cdots x_m}\ar[r]&0.}
$$

For $n$ sequences $\mathbf{a}_i:=(a_{i,1},\ldots ,a_{i,m_i})$ of elements in $R$, we write $\mathbf{a}_1\cdots \mathbf{a}_n$ for the sequence of the $m_1\cdots m_n$ elements $a_{1,i_1}\cdots a_{n,i_n}$ lexicographically ordered. The ideal generated by the elements in $\mathbf{a}_i$ is written $\mathfrak{a}_i$.

The complex $\check\Cc^\bullet_{\mathbf{a}_1,\ldots ,\mathbf{a}_n} (M):=\check\Cc^\bullet_{\mathbf{a}_1} (R)\otimes_R\cdots\otimes_R \check\Cc^\bullet_{\mathbf{a}_n} (R)\otimes_R M $ has as first possibly non-trivial module $\oplus M_{a_{1,i_1}\cdots a_{n,i_n}}=\check\Cc^n_{\mathbf{a}_1\cdots \mathbf{a}_n} (M)$  sitting in cohomological degree $n$  and admits a natural augmentation from $M$ to this first module. The augmented complex:
$$
\xymatrix{0\ar[r]&M\ar[r]& \check\Cc^n_{\mathbf{a}_1,\ldots ,\mathbf{a}_n} (M)\ar^{\phi}[r]&\check\Cc^{n+1}_{\mathbf{a}_1,\ldots ,\mathbf{a}_n} (M)\ar[r] &\cdots}
$$
 is denoted by $\Cc^\bullet_{\mathbf{a}_1,\ldots ,\mathbf{a}_n} (M)$. 

The sequences $0\ra\Cc^{n-1}_{\mathbf{a}_1,\ldots ,\mathbf{a}_n} (M)\ra\Cc^n_{\mathbf{a}_1,\ldots ,\mathbf{a}_n} (M)$ and $0\ra\Cc^0_{\mathbf{a}_1\cdots \mathbf{a}_n} (M)\ra\Cc^1_{\mathbf{a}_1\cdots \mathbf{a}_n} (M)$ are identical, the modules $H^{i+n-1}_{\mathbf{a}_1,\ldots ,\mathbf{a}_n}(M):=H^{i+n-1}(\Cc^\bullet_{\mathbf{a}_1,\ldots ,\mathbf{a}_n} (M))$ and $H^i_{\mathfrak{a}_1\cdots \mathfrak{a}_n}(M):=H^i(\Cc^\bullet_{\mathbf{a}_1\cdots \mathbf{a}_n} (M))$ thus coincide for $i=0$ and are both equal to $H^0_{\mathfrak{a}_1\cdots \mathfrak{a}_n}(M)$.

We write $D_{\mathbf{a}_1,\ldots ,\mathbf{a}_n}(M):=H^n(\check\Cc^\bullet_{\mathbf{a}_1,\ldots ,\mathbf{a}_n} (M))=\ker (\phi )$ and  $D_{\mathfrak{a}_1\cdots \mathfrak{a}_n}(M):=H^1(\check\Cc^\bullet_{\mathbf{a}_1\cdots \mathbf{a}_n} (M))=\ker (\psi )$, with 
 $
\xymatrix{\check\Cc^1_{\mathbf{a}_1\cdots \mathbf{a}_n} (M)\ar^{\psi}[r]&\check\Cc^2_{\mathbf{a}_1\cdots \mathbf{a}_n} (M)}
$
and remark that $D_{\mathfrak{a}_1\cdots \mathfrak{a}_n}(M)\subseteq D_{\mathbf{a}_1,\ldots ,\mathbf{a}_n}(M)$.

The exact sequence of complexes
$$\xymatrix{0\ar[r] & \check\Cc^\bullet_{\mathbf{a}_1,\ldots,\mathbf{a}_n}(M)\ar[r] & \Cc^\bullet_{\mathbf{a}_1,\ldots,\mathbf{a}_n}(M)\ar[r] & M\ar[r] & 0}$$
where $M$ also denotes the complex centered in the $R$-module $M$ at degree $n-1$ gives rise to an exact sequence
\begin{align}
\xymatrix{0\ar[r] & H^{n-1}_{\mathbf{a}_1,\ldots,\mathbf{a}_n}(M)\ar[r] & M\ar[r] & D_{\mathbf{a}_1,\ldots,\mathbf{a}_n}(M)\ar[r] & H^n_{\mathbf{a}_1,\ldots,\mathbf{a}_n}(M)\ar[r] & 0}\label{eq:fundamentalexactsequence}\end{align}
and equalities
$$H^i(\check\Cc^\bullet_{\mathbf{a}_1,\ldots,\mathbf{a}_n}(M))= H^i_{\mathbf{a}_1,\ldots,\mathbf{a}_n}(M)$$ for all $i>n+1$.

\begin{lem}\label{Cechvan}
Assume that $M=H^0_{\mathfrak{a}_1\cdots \mathfrak{a}_n}(M)$, then $\Cc^i_{\mathbf{a}_1\cdots \mathbf{a}_n}(M)=\Cc^{i+n-1}_{\mathbf{a}_1,\ldots ,\mathbf{a}_n}(M)=0$ for any $i>0$.
\end{lem}

\begin{proof}
For $i>0$, summands of $\Cc^i_{\mathbf{a}_1\cdots \mathbf{a}_n}(M)$ or $\Cc^{i+n-1}_{\mathbf{a}_1,\ldots ,\mathbf{a}_n}(M)$ are localizations $M_w$, where $w$ is a multiple of an element of the form $a_{1,i_1}\cdots a_{n,i_n}$, for some $n$-tuple $(i_1,\ldots ,i_n)$.
\end{proof}

Recall that, for any sequence $\mathbf{x}=(x_1,\ldots ,x_m)$ of elements in $R$, $q\in\NN$, $j\in\{1,\ldots ,m\}$  and $\alpha \in H^q_\mathbf{x} (M)$,  $x_j^t \alpha =0$ for some $t$ (see, for instance, the identification in \cite[Theorem 2.3]{H}). This extends to sums of localizations $\Cc^p_\mathbf{a} H^q_\mathbf{x} (M)$, and hence to the subquotients  $H^p_\mathbf{a} (H^q_\mathbf{x} (M))$ or the submodules $D_\mathbf{a} (H^q_\mathbf{x}(M))$, for any sequence $\mathbf{a}$ of elements in $R$. Hence $H^i_{\mathfrak{a}_1\cdots\mathfrak{a}_n}(M)=H^0_{\mathfrak{a}_1\cdots \mathfrak{a}_n}(H^i_{\mathfrak{a}_1\cdots\mathfrak{a}_n}(M))$ for any $i\geq0$, but also:

\begin{lem}\label{suppDab}
For any $R$-module $M$ and any $i\geq0$, $H^i_{\mathbf{a}_1,\ldots,\mathbf{a}_n}(M)=H^0_{\mathfrak{a}_1\cdots \mathfrak{a}_n}(H^i_{\mathbf{a}_1,\ldots,\mathbf{a}_n}(M))$.
\end{lem}

\begin{proof}





The case $n=1$ is the remark before the statement and we induct on $n$. 

The result is clear for $i\leq n-1$ since $H^{n-1}_{\mathbf{a}_1,\ldots,\mathbf{a}_n}(M)=H^0_{\mathfrak{a}_1\cdots \mathfrak{a}_n}(M)$ and $H^{i}_{\mathbf{a}_1,\ldots,\mathbf{a}_n}(M)
=0$ for $i<n-1$.

For $i=n$, by recursion hypothesis and the exact sequence (\ref{eq:fundamentalexactsequence}), given $a_{j,i_j}\in \mathbf{a}_j$ for $1\leq j\leq n$ and $x\in D_{\mathbf{a}_1,\ldots,\mathbf{a}_n}(M)=D_{\mathbf{a}_1,\ldots,\mathbf{a}_{n-1}}(D_{\mathfrak{a}_n}(M))$, there exists $N$ such that $(a_{1,i_1}\cdots a_{n-1,i_{n-1}})^Nx\in D_{\mathfrak{a}_n}(M)$. Hence $(a_{1,i_1}\cdots a_{n,i_n})^{N'}x\in M$ some $N'\geq N$. In other words, $(a_{1,i_1}\cdots a_{n,i_n})^{N'}x=0$ in $H^n_{\mathbf{a}_1,\ldots,\mathbf{a}_n}(M)$. 

For $i\geq n+1$, notice that if $E_t^{\bullet ,\bullet}\Rightarrow H^\bullet$ is a spectral sequence such that $H^0_{\mathfrak{a}_1\cdots \mathfrak{a}_n}(E_t^{p,q})=E_t^{p,q}$, for some $t$ and all $p,q$ with $p+q\geq s$,  then $H^0_{\mathfrak{a}_1\cdots \mathfrak{a}_n}(H^i)=H^i$ for all $i\geq s$. 
 The double complex with $D^{p,q}:=\check\Cc^p_{\mathbf{a}_1,\ldots,\mathbf{a}_{n-1}}(R)\otimes_R\check\Cc^q_{\mathbf{a}_n}(R)\otimes_R M$ gives rise to two spectral sequences that abuts to the homology of $\check\Cc^\bullet_{\mathbf{a}_1,\ldots,\mathbf{a}_n}(M)$.
Taking $t:=n+1$ and $s:=n+2$ in any of these two spectral sequences, it then follows by recursion that 
$H^i_{\mathbf{a}_1,\ldots,\mathbf{a}_n}(M)=H^0_{\mathfrak{a}_1\cdots \mathfrak{a}_n}(H^i_{\mathbf{a}_1,\ldots,\mathbf{a}_n}(M))$ for all $i\geq n+1$ according to what was noticed just above -- also recall that if $\mathfrak a\subseteq \mathfrak b$, then for any module $N$, $H^0_{\mathfrak b}(N)=N$ implies $H^0_{\mathfrak a}(N)=N$.

\end{proof}

\begin{thm}\label{quasiisomorphism}
For any $R$-module $M$ and any $i$,
$$
H^{i+n-1}_{\mathbf{a}_1,\ldots ,\mathbf{a}_n}(M)\simeq H^i_{\mathfrak{a}_1\cdots \mathfrak{a}_n}(M)
$$
and $D_{\mathbf{a}_1,\ldots ,\mathbf{a}_n}(M)\simeq D_{\mathfrak{a}_1\cdots \mathfrak{a}_n}(M)$.
\end{thm}


\begin{proof}
We may assume that $H^0_{\mathfrak{a_1}\cdots \mathfrak{a_n}}(M)=0$ according to Lemma \ref{Cechvan}.

The exact sequences 
$$
0\ra M\ra D_{\mathfrak{a}_1\cdots \mathfrak{a}_n}(M)\ra H^1_{\mathfrak{a}_1\cdots\mathfrak{a}_n}(M)\ra 0
$$
and
$$
0\ra M\ra D_{\mathbf{a}_1,\ldots ,\mathbf{a}_n}(M)\ra H^n_{\mathbf{a}_1,\ldots ,\mathbf{a}_n}(M)\ra 0
$$
provide two isomorphisms of complexes:
$$
\check\Cc^{\bullet}_{\mathbf{a}_1,\ldots ,\mathbf{a}_n}(M)\simeq \check\Cc^{\bullet}_{\mathbf{a}_1,\ldots ,\mathbf{a}_n}(D_{\mathfrak{a}_1\cdots \mathfrak{a}_n}(M))
\quad\mbox{and}\quad
\check\Cc^{\bullet}_{\mathbf{a}_1\cdots \mathbf{a}_n}(M)\simeq \check\Cc^{\bullet}_{\mathbf{a}_1\cdots \mathbf{a}_n}(D_{\mathbf{a}_1,\ldots ,\mathbf{a}_n}(M))
$$
according to  Lemma \ref{Cechvan}, since $H^1_{\mathfrak{a}_1\cdots \mathfrak{a}_n}(M)=H^0_{\mathfrak{a}_1\cdots \mathfrak{a}_n}(H^1_{\mathfrak{a}_1\cdots \mathfrak{a}_n}(M))$, and $H^n_{\mathbf{a}_1,\ldots ,\mathbf{a}_n}(M)=H^0_{\mathfrak{a}_1\cdots \mathfrak{a}_n}(H^n_{\mathbf{a}_1,\ldots ,\mathbf{a}_n}(M))$ by Lemma  \ref{suppDab}. In particular,
\begin{itemize}
\item[(a)] $D_{\mathbf{a}_1,\ldots ,\mathbf{a}_n}(M)=D_{\mathbf{a}_1,\ldots ,\mathbf{a}_n}(D_{\mathfrak{a}_1\cdots \mathfrak{a}_n}(M))$ and $D_{\mathfrak{a}_1\cdots \mathfrak{a}_n}(M)=D_{\mathfrak{a}_1\cdots \mathfrak{a}_n}(D_{\mathbf{a}_1,\ldots ,\mathbf{a}_n}(M))$;

\item[(b)] $H^{i+n-1}_{\mathbf{a}_1,\ldots ,\mathbf{a}_n}(M)=H^{i+n-1}_{\mathbf{a}_1,\ldots ,\mathbf{a}_n}(D_{\mathfrak{a}_1\cdots \mathfrak{a}_n}(M))$ and $H^i_{\mathfrak{a}_1\cdots \mathfrak{a}_n}(M)=H^{i}_{\mathfrak{a}_1\cdots\mathfrak{a}_n}(D_{\mathbf{a}_1,\ldots ,\mathbf{a}_n}(M))$, for any $i\geq 2$.
\end{itemize}

Lemma \ref{suppDab} assures that the two spectral sequences associated to $\check\Cc^\bullet_{\mathbf{a}_1\cdots \mathbf{a}_n}(\check\Cc^\bullet_{\mathbf{a}_1,\ldots ,\mathbf{a}_n}(M))$ collapse at step $2$; it thus provides isomorphisms
$$
H^i_{\mathfrak{a}_1\cdots \mathfrak{a}_n}(D_{\mathbf{a}_1,\ldots ,\mathbf{a}_n}(M))\simeq H^{i+n-1}_{\mathbf{a}_1,\ldots ,\mathbf{a}_n}(D_{\mathfrak{a}_1\cdots \mathfrak{a}_n}(M)),\quad \forall i\geq 2
$$
and shows that $D_{\mathbf{a}_1,\ldots ,\mathbf{a}_n}(D_{\mathfrak{a}_1\cdots \mathfrak{a}_n}(M))=D_{\mathfrak{a}_1\cdots \mathfrak{a}_n}(D_{\mathbf{a}_1,\ldots ,\mathbf{a}_n}(M))$. Thus, by (a), $D_{\mathbf{a}_1,\ldots ,\mathbf{a}_n}(M)=D_{\mathfrak{a}_1\cdots \mathfrak{a}_n}(M)$ and the conclusion follows from (b).
\end{proof}

\begin{cor}\label{quasiisomorphisms}
For any $R$-module $M$, $1\leq i_1<\cdots <i_p\leq n$ with $p\geq 1$ and any $i$,
\begin{itemize}
\item [(a)] $H^{i+p-1}_{\mathbf{a}_{i_1},\ldots ,\mathbf{a}_{i_p}}(M)\simeq H^i_{\mathfrak{a}_{i_1}\cdots \mathfrak{a}_{i_1}}(M)$,
\item [(b)] $H^{i+p-1}(\check\Cc^{\bullet}_{\mathbf{a}_{i_1},\ldots ,\mathbf{a}_{i_p}}(M))\simeq H^i(\check\Cc^{\bullet}_{\mathbf{a}_{i_1}\cdots \mathbf{a}_{i_p}}(M))$.
\end{itemize}
\end{cor}

\section{The four Mayer-Vietoris spectral sequences}\label{fourmvss}

In this section, we make use of Theorem \ref{fourspec} and Theorem \ref{quasiisomorphism} to construct four spectral sequences,
two goes from  cohomology supported in sums of the given ideals to the cohomology supported in the product of all of them, the others from cohomology supported in products of the given ideals to the cohomology supported in the sum of all of them; in each case, the versions include non augmented and augmented \v{C}ech complexes.

To follow usual notation, we set $\check H^i_\mathfrak{a}(M):=H^{i+1}(\check\Cc^\bullet_{\mathbf{a}}(M))$ if $\mathfrak{a}$ is generated by the elements in $\mathbf{a}$.

\begin{thm}\label{mvss}
Let $M$ be an $R$-module and $\mathfrak{a}_i$ be finitely generated ideals, then there exist four convergent spectral sequences
\begin{itemize}

\item [(1a)] $E_1^{n-p,q}=\bigoplus_{i_1<\ldots<i_p}^{1\leq p\leq n}H^q_{\mathfrak{a}_{i_1}+\cdots+\mathfrak{a}_{i_p}}(M)\Rightarrow_p H^{q-(p-1)}_{\mathfrak{a}_1\cdots\mathfrak{a}_n}(M)$,\smallskip

\item [(1b)] $E_1^{n-p,q}=\bigoplus_{i_1<\cdots<i_p}^{1\leq p\leq n}\check H^q_{\mathfrak{a}_{i_1}+\cdots+\mathfrak{a}_{i_p}}(M)\Rightarrow_p \check H^{q-(p-1)}_{\mathfrak{a}_1\cdots\mathfrak{a}_n}(M)$,\smallskip

\item [(2a)] $E_1^{p,q}=\bigoplus_{i_1<\ldots<i_p}^{1\leq p\leq n}H^{q}_{\mathfrak{a}_{i_1}\cdots\mathfrak{a}_{i_p}}(M)\Rightarrow_p H^{q+(p-1)}_{\mathfrak{a}_1+\cdots +\mathfrak{a}_n}(M)$,\smallskip

\item [(2b)] $E_1^{p,q}=\bigoplus_{i_1<\ldots<i_p}^{1\leq p\leq n}\check H^q_{\mathfrak{a}_{i_1}\cdots\mathfrak{a}_{i_p}}(M)\Rightarrow_p \check H^{q+(p-1)}_{\mathfrak{a}_1+\cdots +\mathfrak{a}_n}(M)$.
\end{itemize}
\end{thm}

\begin{proof}
We will apply Theorem \ref{fourspec}  to the $n$-multicomplex $C^\bb=\Cc^\bullet_{\mathbf{a}_1}(\Cc^\bullet_{\mathbf{a}_2}(\cdots (\Cc^\bullet_{\mathbf{a}_n}(M))))$ for (1a) and (2a) and to the punctured complex  $C^\bb_P$ for (1b) and (2b). Recall that the $i$-th homology of the totalization of
$
C^{\bb}_{\Face_{i_1,\ldots ,i_p}}=\Cc^\bullet_{\mathbf{a}_{i_1}}(\Cc^\bullet_{\mathbf{a}_{i_2}}(\cdots (\Cc^\bullet_{\mathbf{a}_{i_p}}(M))))
$
is $H^i_{\mathfrak{a}_{i_1}+\cdots+\mathfrak{a}_{i_p}}(M)$, while the one of $C^{\bb}_{\Pface_{i_1,\ldots ,i_p}}=(C^{\bb}_\Pface)_{\Face_{i_1,\ldots ,i_p}}
=(C^{\bb}_{\Face_{i_1,\ldots ,i_p}})_\Pface$ is $\check H^{i-1}_{\mathfrak{a}_{i_1}+\cdots+\mathfrak{a}_{i_p}}(M)$.

By Corollary \ref{quasiisomorphisms} (a), $H^{i}({^+{C}}^{\bullet}_{\intF_{i_1,\ldots ,i_p}})=H^{i-n+1}_{\mathfrak{a}_{i_1}\cdots \mathfrak{a}_{i_p}}(M)$, while $H^{i}(C^{\bullet}_{\intF_{i_1,\ldots ,i_p} })=\check H^{i-n}_{\mathfrak{a}_{i_1}\cdots \mathfrak{a}_{i_p}}(M)$ by Corollary \ref{quasiisomorphisms} (b). 

Therefore (1a), (1b), (2a) and (2b) follow respectively from items (2), (1), (4) and (3) in Theorem \ref{fourspec}.
\end{proof}

Items (1b) and (2b) concerns local cohomology in the version that equals sheaf cohomology: $\check H^i_{\mathfrak a}(M)=H^i(U,\tilde M)$ with $U:=\Spec (R)\setminus V({\mathfrak a})$ by \cite[1.2.3 \& 1.4.3]{EGAIII}. In the other two versions (1a) and (2a),  \v Cech cohomlogy is related to sheaf cohomology by the exact sequence 
$$
\xymatrix{
0\ar[r]&H^0_{\mathfrak a}(M)\ar[r]&M\ar[r]&H^0(U,\tilde M )\ar[r]&H^1_{\mathfrak a}(M)\ar[r]&0,}
$$
as proved in the same reference, and $H^i_{\mathfrak a}(M)=\check H^{i-1}_{\mathfrak a}(M)=H^{i-1}(U,\tilde M)$ for $i\geq 2$. Also recall that $H^i_{\mathfrak a}(-)$ is the $i$-th right derived functor of $H^0_{\mathfrak a}(-)$ whenever $R$ is Noetherian.\\

To give a hint on the behavior of the spectral sequences above, we make a few comments on the spectral sequence (1a); is construction is similar to (1b). For concrete examples and applications, that partially inspired this work, see \cite{AGZ,Lyu,Hol}. In general, its differential $d_r$ has degree $(r,1-r)$, i.e. $d_r^{n-p,q}:E_r^{n-p,q}\ra E_r^{n-p+r,q-r+1}$ for all $p,q$.

For three ideals it has only three columns, so that it degenerates at the third page. A sketch of such a spectral sequence, where the blue arrows represent first-page differentials and the  red ones represent the directions of the second-page differentials, is the following:

$$
\xymatrix@=2em{&\vdots&\vdots&\vdots&
\\
0 \ar@[blue][r]& H^{j+2}_{\mathfrak{a}_1+\mathfrak{a}_2+\mathfrak{a}_3}(M)\ar@[blue]^{\varphi}[r]\ar@[red][rrd]\ar@{--}[rrdd] &  \oplus_{i< j}H^{j+2}_{\mathfrak{a}_i+\mathfrak{a}_j}(M)\ar@[blue]^{\psi}[r] & \oplus_i H^{j+2}_{\mathfrak{a}_i}(M) \ar@[blue][r]& 0
\\
0\ar@[blue][r] & H^{j+1}_{\mathfrak{a}_1+\mathfrak{a}_2+\mathfrak{a}_3}(M)\ar@[blue]^{\varphi '}[r]\ar@[red][rrd] &  \oplus_{i< j}H^{j+1}_{\mathfrak{a}_i+\mathfrak{a}_j}(M)\ar@[blue]^{\psi '}[r] & \oplus_i H^{j+1}_{\mathfrak{a}_i}(M)\ar@[blue][r]& 0
\\
0 \ar@[blue][r]& H^{j}_{\mathfrak{a}_1+\mathfrak{a}_2+\mathfrak{a}_3}(M)\ar@[blue]^{\varphi ''}[r] &  \oplus_{i< j}H^{j}_{\mathfrak{a}_i+\mathfrak{a}_j}(M)\ar@[blue]^{\psi ''}[r] & \oplus_i H^{j}_{\mathfrak{a}_i}(M) \ar@[blue][r]& 0
\\
&\vdots&\vdots&\vdots&
}
$$

The names of the $E_1^{\bullet ,\bullet}$ differentials are simplified to improve clarity (e.g. $\varphi =d_1^{0,j+2}$); these are given by the natural maps coming from the inclusions of ideals (e.g. the inclusion $\mathfrak{a}_1\subseteq \mathfrak{a}_1+\mathfrak{a}_2$ provides a natural map $H^i_{\mathfrak{a}_1+\mathfrak{a}_2}(-)\ra H^i_{\mathfrak{a}_1}(-)$ for every $i$), each of these pieces taken with an appropriate sign.

The dotted line is the diagonal where the filtration $0\subseteq F^2\subseteq F^1\subseteq F^0 =H^j_{\mathfrak{a}_1\mathfrak{a}_2\mathfrak{a}_3}(M)$ is given by the infinity terms (here isomorphic to the terms at step 3). This filtration satisfies:
$$F^2\simeq\coker(\ker (\varphi '){\color{red}\ra}\coker (\psi '')), \quad F^1/F^2\simeq \frac{\ker(\psi ')}{\mbox{im}(\varphi ')},\quad F^0/F^1\simeq\ker(\ker (\varphi ){\color{red}\ra}\coker (\psi ')).$$


\begin{rem}
\begin{itemize}


\item [(a)] In Theorem \ref{mvss} (1a) we retrieve the Mayer-Vietoris spectral sequence in \cite{Lyu}.

\item [(b)] A spectral sequence as (2b) in Theorem \ref{mvss} could be obtained as a \v{C}ech spectral sequence, by  \cite[Th\'eor\`eme 5.4.1]{God}. Indeed, the first terms in the quoted spectral sequence are described in \cite[Chap. II, \S 5.3]{God} as sheaf cohomologies on open sets that are intersections of the open complements $U_i$ of $V(\mathfrak{a_i})$ in $\Spec (R)$ : the ones that appear in the \v{C}ech covering of $\cup_i U_i =\Spec (R)\setminus V({\mathfrak{a}}_1 +\cdots +{\mathfrak{a}}_n)$. These sheaf cohomology modules are in turn isomorphic to local cohomologies with support in products of the ideals, by \cite[1.2.3 \& 1.4.3]{EGAIII}, since $U_{i_1}\cap\cdots\cap U_{i_p}=\Spec (R)\setminus V({\mathfrak{a}}_{i_1}\cdots {\mathfrak{a}}_{i_p})$. 

\item [(c)] Such spectral sequences degenerate into the celebrated Mayer-Vietoris long exact sequence when one consider only two ideals. It is a different way from \cite[Theorem 9.4.3]{SS} and \cite{T} of obtaining this exact sequence.

\end{itemize}
\end{rem}

\noindent{\bf Acknowledgements.} First and third named authors thanks France-Brazil network RFBM for supporting this work. The second named author was supported by a CAPES Doctoral Scholarship.

\end{document}